\documentclass[10pt]{amsart}
\usepackage{comment}
\usepackage{amsmath,amsfonts,euscript,amscd,amsthm,amssymb,upref,graphics}
\usepackage{mathrsfs,yhmath}
\usepackage[all]{xy}
\usepackage[unicode,pdfencoding=auto,psdextra,hidelinks]{hyperref}
\usepackage{amsrefs}

\calclayout

\usepackage{latexsym}
\usepackage{graphicx}
\usepackage{tikz}
\usepackage{caption}
\usepackage{subcaption}
\usepackage{enumitem}
\usepackage{tikz-cd}

\DeclareMathOperator{\rank}{rank}

\usepackage{bm}
\usepackage{pgfplots}

\newtheorem{theo}{Theorem}[section]

\newtheorem{coro}[theo]{Corollary}
\newtheorem{lemm}[theo]{Lemma}
\newtheorem{prop}[theo]{Proposition}

\theoremstyle{remark}
\newtheorem{rema}[theo]{Remark}
\newtheorem{defi}[theo]{Definition}

\title{Unitriangular factorization of holomorphic symplectic matrices}
\author[G. Huang]{Gaofeng Huang}
\author[F. Kutzschebauch]{Frank Kutzschebauch}
\author[B. Tran]{Phan Quoc Bao Tran}
\address{Mathematisches Institut \\ University of Bern \\  
Bern, Switzerland}
\email{gaofeng.huang@unibe.ch, frank.kutzschebauch@unibe.ch, phan.tran@unibe.ch}


\begin{document}

\maketitle

\begin{abstract}
We prove that every holomorphic 
symplectic matrix can be factorized as a product of holomorphic unitriangular matrices with respect to the symplectic form $ \left[\begin{array}{ccc} 0 & L_n \\ -L_n & 0\end{array}\right]$ 
where $L$ is the $n \times n$ matrix with $1$ along the skew-diagonal.
Also we prove that holomorphic unitriangular matrices with respect to this symplectic form  are products of not more than $7$ holomorphic unitriangular matrices with respect to the standard symplectic form $\left[\begin{array}{ccc} 0 & I_n \\ -I_n & 0\end{array}\right]$, thus solving an open problem posed in \cite{HKS}.
Combining these two results allows for estimates of the optimal number of factors in the factorization by  holomorphic unitriangular matrices with respect to the standard symplectic form. The existence of that factorization was  obtained earlier by Ivarsson-Kutzschebauch and Schott, however without any estimates. 
Another byproduct of our results is a new, much less technical and more elegant proof of this factorization.
\end{abstract}

\tableofcontents
\section{Introduction}
Factorization of a given matrix with elements in a field  has been studied widely with many applications to different fields of sciences. The most well-known one is the factorization of a matrix $A \in \mathrm{SL}_n(\mathbb{C})$ into product of some elementary matrices of the form $E + \alpha e_{ij}$, $i \neq j$, called elementary matrices. This factorization is equivalent to factorization into so-called unitriangular matrices (triangular matrices with 1s on the diagonal). Another interesting class of matrices is the class of symplectic matrices.  
Generalizing the case of matrices with entries from a field, it is much more challenging to consider matrices with entries from a ring, the subject of $K$-theory. In this paper we always work
 with  very specific rings, namely the ring of holomorphic functions on a Stein space or the ring of continuous functions on a normal finite dimensional topological space. Recall that the  unitriangular symplectic matrices  with respect to the standard symplectic form
\begin{align*}
    \Omega = \left[\begin{array}{ccc}
        0 & I_n \\
        -I_n & 0
    \end{array}\right]
\end{align*}
are of the following form
\begin{center}
     $M^-(G) = \left[\begin{array}{ccc} I_n & 0 \\ G & I_n\end{array}\right]$ and $M^+(G)=\left[\begin{array}{ccc} I_n & G \\ 0 & I_n\end{array}\right]$, \text{ where $G$ is symmetric.} 
\end{center}
Since the entries of $G$ are holomorphic functions on a reduced Stein space, it can be understood as a holomorphic map:
\begin{center}
    $G: X \to \mathbb{C}^{\frac{n(n+1)}{2}}$.
\end{center}
In \cite{IKL} (for $n=2$)  and \cite{Scho} (for general $n$), the following theorem about factorizing matrices into product of  unitriangular symplectic matrices  has been proved:

\begin{theo}\label{old-theo}
There exists a natural
number $K = K(n, d)$ such that given any finite dimensional reduced Stein space $X$ of dimension
$d$ and any null-homotopic holomorphic map $f : X \to \mathrm{Sp}_{2n}(\mathbb{C})$ there exists a holomorphic map:
\begin{center}
$G = (G_1, \cdots, G_K) : X \to (\mathbb{C}^{\frac{n(n+1)}{2}} )^K
$    
\end{center}
such that
\begin{center}
 $f(x) = M^-(G_1(x)) M^+(G_2(x)) 
 \cdots M^{\pm}(G_K(x)),$ \text{ where $x \in X$.}    
\end{center}
\end{theo}
It is immediate that any product of unitriangular symplectic matrices is connected by a path to the constant identity matrix $I_{2n}$, by multiplying the off-diagonal entries by $t \in [0, 1]$, i.e.
\begin{center}
    $\left[\begin{array}{ccc} I_n & tG \\ 0 & I_n\end{array}\right]$ and $\left[\begin{array}{ccc} I_n & 0 \\ tG & I_n\end{array}\right]$.
\end{center}
Therefore the requirement of null-homotopy of the map $f$ is necessary.

In our paper, we study factorization of a holomorphic symplectic matrix into unitriangular symplectic factors with respect to the symplectic form 
\begin{center}
    $\tilde{\Omega} = \left[\begin{array}{ccc}
        0 & L_n \\
        -L_n & 0
    \end{array}\right]$, where $L$ is the $n \times n$ matrix with $1$ along the skew-diagonal.
\end{center}
The lower unitriangular $\tilde{\Omega}$-symplectic factors are the $\Omega$-symplectic matrices of the form
\begin{center}
    $\left[\begin{array}{ccc} A^{-T} & 0 \\ B & A\end{array}\right],$  $B A^T$ symmetric and $A$ upper unitriangular,
\end{center} 
In order to introduce convenient coordinates $(A,Z) \in \mathbb{C}^{\frac{n(n-1)}{2}} \times  \mathbb{C}^{\frac{n(n+1)}{2}}$ on this set of matrices we  will write them as
\begin{center}
    $M^{-}(A,Z)  = \left[\begin{array}{ccc} I_n & 0 \\ Z & I_n\end{array}\right] \left[\begin{array}{ccc} A^{-T} & 0 \\ 0 & A\end{array}\right]= \left[\begin{array}{ccc} A^{-T} & 0 \\ B & A\end{array}\right],$ $Z = B A^T$ symmetric and $A$ upper unitriangular.
\end{center}
The same for upper unitriangular
\begin{center}
    $ M^{+}(A,Z)  = \left[\begin{array}{ccc} I_n & Z \\ 0 & I_n\end{array}\right] \left[\begin{array}{ccc} A^{-1} & 0 \\ 0 & A^T \end{array}\right] = \left[\begin{array}{ccc} A^{-1} & C \\ 0 & A^{T}\end{array}\right],$ $ Z = C A^{-T}$ symmetric and $A$ upper unitriangular.
\end{center}

Throughout the paper, we call these matrices the \textbf{elementary matrices}. 
Our main theorem is a solution to the Gromov-Vaserstein problem for the symplectic form $\tilde \Omega$.

\begin{theo}\label{main-theo} 
There exists a natural
number $K = K(n, d)$ such that given any finite dimensional reduced Stein space $X$ of dimension
$d$ and any null-homotopic holomorphic map $f : X \to \mathrm{Sp}_{2n}(\mathbb{C})$ there exist holomorphic maps:
\begin{center}
$A = (A_1, \cdots, A_K) : X \to (\mathbb{C}^{\frac{n(n-1)}{2}} )^K
$  and $Z = (Z_1, \cdots, Z_K) : X \to (\mathbb{C}^{\frac{n(n+1)}{2}} )^K
$ 
\end{center}
such that
\begin{center}
 $f(x) = M^-(A_1(x), Z_1(x)) M^+(A_2(x), Z_2(x)) 
 \cdots M^{\pm}(A_K(x), Z_K(x)),$ \text{ where $x \in X$.}    
\end{center}
\end{theo}

This theorem together  with our second main result, saying that the unitriangular $\tilde{\Omega}$-symplectic matrices from Theorem \ref{main-theo} are products of not more than 8 
unitriangular $\Omega$-symplectic matrices, gives a new proof of Schott's result.  Moreover our proof is much more elegant and much easier to follow. In fact we use a novel strategy to prove the stratified ellipticity of a certain map, which is the core of the proof. More precisely, we can reduce the number of equations for the fibers of this map, given from the beginning by $2n$ polynomial equations, to one single equation. From this equation
the stratified ellipticity is as easily obtained as in the work of Ivarsson and Kutzschebauch \cite{KI}. In Schott's proof (and in the preceding work of Ivarsson-Kutzschebauch-L\o w)
the reduction is only done from $2n$ to $n$ equations. Studying fibers given by $n$ polynomial equations is the  hard and unpleasant technical part of the papers \cite{IKL}, \cite{Scho}.

We denote by $K_{symp}(n,d)$ the minimal number
of unitriangular factors for $2n \times 2n$ $\Omega$-symplectic matrices with holomorphic functions as entries in any $d$-dimensional Stein space and
$\tilde K_{symp}(n,d)$ the same for  $\tilde \Omega $-symplectic  factors. Observe that the optimal number of factors over the field of complex numbers is classically known in the $\tilde\Omega$-case over fields (and rings of Bass stable rank 1)  to be $4$, however, the optimal number for the $\Omega$-case has been only very recently found to be $5$ by  Jin, Tang and Zhu  in \cite{JTZ}. As  such proofs over the complex numbers usually do, their proof depends on Jordan normal form. It is the fact that Jordan normal form is discontinuous making the problem for continuous and holomorphic matrices so delicate. However, we can reduce the problem to a simpler case by multiplying the unitriangular matrix $A$ with a diagonal matrix $\Lambda$ with pairwise different elements on the diagonal. This new matrix $\Lambda A$ is diagonalizable and the diagonalization process is algebraic. From that point on  we can apply the factorization technique in \cite{JTZ} to get the optimal results as follows:

\begin{theo}\label{main-theo2}
For $n \ge 4$ we have
$$6 \le K_{symp}(n,d) \le 7 \tilde K_{symp}(n,d).$$
For matrices of smaller sizes, there is a better upper bound
\begin{align} \label{K-Ktilde}
    K_{symp}(n,d) \le 4 \tilde K_{symp}(n,d), \quad \text{ for } n = 2 \text{ and } n=3.
\end{align}
\end{theo}
This theorem solves the problem posed in \cite[p. 6]{HKS}.
Moreover, combining Theorem \ref{main-theo2}
with $K$-theoretic results from \cite{VSS},\cite{HKS}
we are in certain cases able to give explicit bounds for the number of factors in Schott's theorem as follows:

\begin{prop} \label{bounds}
    The number of factors for $N \ge n+1$ is bounded by the number for $n$:
    \begin{align} \label{K(N,d)}
        K_{symp}(N,d) \le 7 \tilde{K}_{symp}(n,d).  
    \end{align}
    When the Stein space $X$ is one-dimensional,
    \begin{align}
        5 &\le K_{symp} (n,1)  \le 16 \ \  \text{for} \ \ n=2,3, \label{K(n,1)n=2,3} \\
   6 &\le K_{symp} (n,1)  \le 28 \ \ \text{for } \  n \ge 4.  \label{K(n,1)n>=4}
    \end{align}
    When the Stein space $X$ is two-dimensional,
    \begin{align}   
   5 &\le K_{symp} (n,2)  \le 20 \ \ \text{for} \ \ n=2,3, \label{K(n,2)n=2,3}\\
   6 &\le   K_{symp}(n,2) \le  35  \ \  \text{for } \ n \ge 4. \label{K(n,2)n>=4}
  \end{align}
\end{prop}

\subsection{Strategy of the proof}
    We consider the product of $K$ unitriangular symplectic matrices
    \begin{center}
    $\Psi_K = M^-(A_1,Z_1) M^+(A_2, Z_2) \cdots  M^{\pm}(A_K, Z_{K})$. 
    \end{center} 
    If we naturally identify the set of upper unitriangular matrices with the set $\mathbb{C}^{\frac{n(n-1)}{2}}$, and the set of symmetric matrices with the set $\mathbb{C}^{\frac{n(n+1)}{2}}$, we can consider $\Psi_K$ as a map
    \begin{center}
    $\Psi_K: (\mathbb{C}^{\frac{n(n-1)}{2}} \times \mathbb{C}^{\frac{n(n+1)}{2}})^K \to \mathrm{Sp}_{2n}(\mathbb{C})$.
    \end{center}
     
    
Given a holomorphic map $f: X \to \mathrm{Sp}_{2n}(\mathbb{C})$ from a finite dimensional Stein space $X$ to the symplectic matrices, we need to find a holomorphic lift $F$ such that the following diagram commutes 
\begin{center}
    \begin{tikzcd}
    & & & (\mathbb{C}^{\frac{n(n-1)}{2}} \times \mathbb{C}^{\frac{n(n+1)}{2}})^K \arrow[dd,"\Psi_K"] \\ \\
    & X \arrow[uurr,"F", dashed] \arrow[rr,"f"] & & \mathrm{Sp}_{2n}(\mathbb{C})
\end{tikzcd}
\end{center}
It is very hard to study the map $\Psi_K$ directly, therefore we choose to consider its composition with the projection $\pi_{2n}$ of a $2n \times 2n$ matrix to its last row
\begin{align}   \label{def-PhiK}
    \Phi_K = \pi_{2n} \circ \Psi_K : (\mathbb{C}^{\frac{n(n-1)}{2}} \times \mathbb{C}^{\frac{n(n+1)}{2}} )^K \to \mathbb{C}^{2n}\backslash \lbrace 0 \rbrace.  
\end{align}
By finding a holomorphic lift in 
\begin{center}
    \begin{tikzcd}
    & & & (\mathbb{C}^{\frac{n(n-1)}{2}} \times \mathbb{C}^{\frac{n(n+1)}{2}})^K \arrow[dd,"\Phi_K"] \\ \\ 
    & X \arrow[rr,"\pi_{2n} \circ f"] \arrow[uurr, dashed] &  & \mathbb{C}^{2n}\setminus \{ 0\}
    \end{tikzcd}
\end{center}
we obtain a product of holomorphic elementary matrices whose last row is the same as the last row of $f$. Using standard linear algebra arguments, Theorem \ref{main-theo} follows by an induction on the size of the matrices. 

To find this holomorphic lift, we will use the Oka principle for stratified elliptic submersions. In fact, after removing the singularity set $S_K$ from its domain, the map $\Phi_K$ becomes such. 
\begin{prop} \label{prop: Phi-K}
The map $\Phi_K: (\mathbb{C}^{\frac{n(n-1)}{2}} \times \mathbb{C}^{\frac{n(n+1)}{2}})^K \setminus S_K \to \mathbb{C}^{2n}\backslash \lbrace 0 \rbrace$ is a stratified elliptic submersion.    
\end{prop}
The continuous lift is constructed using the continuous factorization of symplectic matrices in \cite{IKL}.

The proof of Proposition \ref{prop: Phi-K} uses complete holomorphic vector fields tangent to the fibers of the map $\Phi_K$. We are able to reduce the $2n$ equations describing these fibers, to a single equation. This reduction in the number of defining equations greatly simplifies the search for a set of complete vector fields spanning the tangent space of the fiber at each point.

\section{The Oka principle for stratified elliptic submersions}


In this section, we recall the Oka principle that we will use as a crucial tool in our proof of Theorem \ref{main-theo} (see \cite{For1, Gro}), following the monograph of Franc Forstneri{\v c} \cite{For2}. 

Let $h\colon Z \to X$ be a holomorphic submersion of a complex manifold $Z$ onto a complex manifold $X$. For any $x\in X$ the fiber over $x$ of this submersion will be denoted by $Z_x$. At each point $z\in Z$ the tangent space $T_zZ$ contains the vertical tangent space $VT_zZ=\ker Dh$. For holomorphic vector bundles $p\colon E \to Z$ we denote the zero element in the fiber $E_z$ by $0_z$.



\begin{defi}
Let $X$ and $Z$ be complex spaces. A holomorphic map $h: Z \to X$ is said to be a submersion if for each point $z_0\in Z$ it is locally equivalent via a fiber preserving biholomorphic map to a projection $p\colon U\times V \to U$, where $U\subset X$ is an open set containing $h(z_0)$ and $V$ is an open set in some $\mathbb{C}^d$.  
\end{defi}



\begin{defi}
Let $h\colon Z \to X$ be a holomorphic submersion of a complex manifold $Z$ onto a complex manifold $X$. A fiber spray on $Z$ associated with $h$ is a triple $(E,p,s)$, where $p\colon E\to Z$ is a holomorphic vector bundle and $s\colon E\to Z$ is a holomorphic map such that for each $z\in Z$ we have 
\begin{itemize}
\item[(i)]{$s(E_z)\subset Z_{h(z)}$,}
\item[(ii)]{$s(0_z)=z$, and}
\item[(iii)]{the derivative $Ds(0_z)\colon T_{0_z}E\to T_zZ$ maps the subspace $E_z\subset T_{0_z}E$ surjectively onto the vertical tangent space $VT_zZ$.}
\end{itemize} 
\end{defi}

\begin{rema} \label{VF-spray}
    One way of constructing dominating fiber sprays, as pointed out by Gromov, is to find finitely many $\mathbb{C}$-complete vector fields that are tangent to the fibers and span the tangent space of the fibers at all points in $Z$. One can then use the flows $\varphi_j^t$ of these vector fields $V_j$ to define $s\colon Z\times \mathbb{C}^N\to Z$ via $s(z,t_1,\dots, t_N)=\varphi_1^{t_1}\circ \dots \circ \varphi_N^{t_N}(z)$ which gives a spray.
\end{rema}

\begin{defi}
We say that a holomorphic submersion $h: Z \to X$ is \textbf{stratified elliptic} if there is a descending
chain of closed complex subspaces $X = X_m \supset \cdots \supset X_0$ such that each stratum $Y_k = X_k \backslash X_{k-1}$
is regular and the restricted submersion $h: Z|_{Y_k} \to Y_k$ admits a dominating fiber spray over a small neighborhood of any point $x \in Y_k$.
\end{defi}

Let $\pi: E \to B$ be a holomorphic map of a complex space $E$ onto a complex space $B$. (All complex spaces are assumed to be reduced.) Assume that $X$ is a Stein space, $P_0 \subset P$ are compact Hausdorff spaces (the parameter spaces), and $f : P \times X \to B$ is a continuous $X$-holomorphic map, meaning that the map $f(p, \cdot): X \to B$ is holomorphic for every fixed $p \in P$. Assume that $F_0: P \times X \to E$ is a continuous lifting of $f$, i.e., $\pi \circ F_0 = f$, which is $X$-holomorphic over $P_0$. We are looking for a homotopy of liftings $F_t : P \times X \to E$ of $f$ ($t \in [0,1]$) which is fixed on $P_0 \times X$, such that $F_1 = F$ is an $X$-holomorphic lifting. The situation is illustrated with the following diagram:

\[
\begin{tikzcd}
& P_0 \times X \arrow[r, "F_0"] \arrow[d, hook, "incl"'] & E \arrow[d, "\pi"]   \\
& P \times X \arrow[ur, dashed, "F"] \arrow[r, "f"'] & B
\end{tikzcd}
\]

We say that the map $\pi: E \to B$ satisfies the \textbf{parametric Oka property} if such a homotopy $F_t$ exists for any data $(X, P_0, P, f, F_0)$. 
It can be illustrated by the following commutative diagram:

\[
\begin{tikzcd}
P_0 \arrow[d, hook, "incl"] \arrow[rr, "F_0"] & & \mathcal{O}(X, E)  \arrow[d, ""] \arrow[rr, hook, "incl"]  & & \mathcal{C}(X, E) \arrow[d, "\pi"] \\
P \arrow[urr, "F_1"] \arrow[urrrr, "\quad F_0"'] \arrow[rr, "f"'] & &\mathcal{O}(X, B) \arrow[rr, hook, "incl"'] & & \mathcal{C}(X, B)
\end{tikzcd}
\]


Note that the problem of lifting a holomorphic map $f : X \to B$ to a holomorphic map $F : X \to E$ reduces to the problem of finding a section of the pullback of $\pi: E \to B$ by $f$. 

\begin{theo} \label{POP}
    Every stratified elliptic submersion enjoys the parametric Oka property.
\end{theo} 

\section{Proof of main theorem}\label{3}
In this section we explain the inductive proof of a parametric version of Theorem \ref{main-theo}. 

\begin{theo}\label{main-theo-parametric} 
There exists a natural
number $K = K(n, d)$ such that given any compact Hausdorff space $P$, any finite dimensional reduced Stein space $X$, such that $P \times X$ has covering dimension $d$, and any null-homotopic holomorphic map $f : P \times X \to \mathrm{Sp}_{2n}(\mathbb{C})$, there exist $X$-holomorphic maps:
\begin{center}
$A = (A_1, \cdots, A_K) : P \times X \to (\mathbb{C}^{\frac{n(n-1)}{2}} )^K
$  and $Z = (Z_1, \cdots, Z_K) : P \times X \to (\mathbb{C}^{\frac{n(n+1)}{2}} )^K
$ 
\end{center}
such that
\begin{center}
 $f(p, x) = M^-(A_1(p,x), Z_1(p,x)) M^+(A_2(p,x), Z_2(p,x)) 
 \cdots M^{\pm}(A_K(p,x), Z_K(p,x)),$ \text{ where $p \in P, x \in X$.}    
\end{center}
\end{theo}

We abbreviate an element $((A_1,Z_1),(A_2,Z_2),\cdots,(A_K,Z_K)) \in (\mathbb{C}^{\frac{n(n-1)}{2}} \times \mathbb{C}^{\frac{n(n+1)}{2}})^K$ by $(\Vec{A}_K,\Vec{Z}_K )$. 
A lengthy but straightforward calculation yields the following.

\begin{prop} \label{SingSet}
    The singularity set $S_K$ of $\Phi_K: (\mathbb{C}^{\frac{n(n-1)}{2}} \times \mathbb{C}^{\frac{n(n+1)}{2}})^K  \to \mathbb{C}^{2n}\backslash \lbrace 0 \rbrace$ defined in \eqref{def-PhiK} is given by
    \begin{align*}
        S_K = \lbrace (\Vec{A}_K,\Vec{Z}_K)|Z_{i} e_n = 0, i<K;  A_j e_n = e_n, 1 < j < K \rbrace.
    \end{align*}
\end{prop}
In other words, in the singular set the last columns of all matrices $Z_i$, except $Z_K$, should be 0. Also, the last columns of all $A_i$, except $A_1$ and $A_K$, should be $e_n$.

The following result will be proven in Section \ref{sec-stra}:
\begin{prop} \label{reg-locus}
For $n \ge 2, K \ge 3$, the map $\Phi_K: (\mathbb{C}^{\frac{n(n-1)}{2}} \times \mathbb{C}^{\frac{n(n+1)}{2}})^K \setminus S_K \to \mathbb{C}^{2n}\backslash \lbrace 0 \rbrace$ is a stratified elliptic submersion.    
\end{prop}

\begin{prop}\label{sup-main} 
Given a null-homotopic, continuous $X$-holomorphic map $f : P \times X \to \mathrm{Sp}_{2n}(\mathbb{C})$, where $P$ is a compact Hausdorff space and $X$ a finite dimensional reduced Stein space. Assume that there exist a natural number $L$ and a continuous lift $F : P \times X \to (\mathbb{C}^{\frac{n(n-1)}{2}} \times \mathbb{C}^{\frac{n(n+1)}{2}})^L$ of $f$ along $\Psi_L$, i.e., $\Psi_L \circ F = f$.
Then there exist a continuous lift $\tilde F : P \times X \to (\mathbb{C}^{\frac{n(n-1)}{2}} \times \mathbb{C}^{\frac{n(n+1)}{2}})^{L+2} \setminus S_{L+2}$ of $f$ along $\Psi_{L+2}$, avoiding the singularity set and a continuous homotopy $h_t: P \times X \to (\mathbb{C}^{\frac{n(n-1)}{2}} \times \mathbb{C}^{\frac{n(n+1)}{2}})^{L+2}\setminus S_{L+2}$ of lifts of $\pi_{2n} \circ f$ along $\Phi_{L+2}$, such that $h_0 = \pi_{2n} \circ \tilde F $ and $h_1$ is $X$-holomorphic.
\end{prop}
\begin{proof}
By assumption we have
\begin{center}
    $F = (F_1, \cdots, F_L) : P \times X \to (\mathbb{C}^{\frac{n(n-1)}{2}} \times \mathbb{C}^{\frac{n(n+1)}{2}} )^L$ with $f = \Psi_L \circ F, \, F_k =  (A_k,Z_k)$.   
\end{center}
We define the map
$\tilde F : P \times X \to (\mathbb{C}^{\frac{n(n-1)}{2}} \times \mathbb{C}^{\frac{n(n+1)}{2}} )^{L+2}$
by
 \begin{center}
  $\tilde F = ((I,I), (I,0), F^{\prime}_1, F_2, \cdots, F_L)$,  where $F_1'=(A_1, Z_1-I)$.
 \end{center}
Since
\begin{center}
    $ \left[\begin{array}{ccc} I & 0 \\ Z_1 & I\end{array}\right]\left[\begin{array}{ccc} A_1^{-T} & 0 \\ 0 & A_1\end{array}\right] = \left[\begin{array}{ccc} I & 0 \\ I & I\end{array}\right]\left[\begin{array}{ccc} I & 0 \\ 0 & I\end{array}\right] \left[\begin{array}{ccc} I & 0 \\ Z_1-I & I\end{array}\right]\left[\begin{array}{ccc} A_1^{-T} & 0 \\ 0 & A_1 \end{array}\right] $
\end{center}
we have $f = \Psi_{L+2} \circ \tilde F$. By Proposition \ref{SingSet}, $\tilde{F}$ has image outside the singularity set $S_{L+2}$. Then Proposition \ref{reg-locus} and Theorem \ref{POP} give a continuous homotopy $h_t: P \times X \to (\mathbb{C}^{\frac{n(n-1)}{2}} \times \mathbb{C}^{\frac{n(n+1)}{2}})^{L+2}\setminus S_{L+2}$ of liftings of $\pi_{2n} \circ f$ along $\Phi_{L+2}$, such that $h_0 = \pi_{2n} \circ \tilde F $ and $h_1$ is $X$-holomorphic.
\end{proof}

\begin{proof}[Proof of Theorem \ref{main-theo-parametric}]
The proof is by induction on $n$. The induction base follows from the factorization for $\mathrm{Sp}_2(\mathbb{C}) = \mathrm{SL}_2(\mathbb{C})$ in \cite{KI}.

 Assume that the theorem is true for $n-1$, we proceed to prove it for $n$. By the solution to the continuous Vaserstein problem, there exist a number $L$ and a continuous lift $F : P \times X \to (\mathbb{C}^{\frac{n(n-1)}{2}} \times \mathbb{C}^{\frac{n(n+1)}{2}})^L$ of $f$ along $\Psi_L$, i.e., $\Psi_L \circ F = f$. Then Proposition \ref{sup-main} yields a continuous homotopy $h_t: P \times X \to (\mathbb{C}^{\frac{n(n-1)}{2}} \times \mathbb{C}^{\frac{n(n+1)}{2}})^{L+2}\setminus S_{L+2}$ of lifts of $\pi_{2n} \circ f$ along $\Phi_{L+2}$, such that $h_0 = \pi_{2n} \circ \tilde F $ and $h_1$ is $X$-holomorphic, where $\tilde{F}$ is as in the proposition. 
 
Since the two symplectic matrices $f(p,x)$ and $\Psi_{L+2}(h_t(p,x))$ have the same last row, we get
\begin{center}
     $\Psi_{L+2}(h_t(p,x))f(p,x)^{-1} = \left[\begin{array}{cccc} A_{t}(p, x)  & 0 & B_{t}(p, x) & b_{1,t}(p,x) \\  a_{t}(p,x) & 1 & b_{2,t}(p,x) & b_{3,t}(p,x) \\
    C_t(p,x) & 0 & D_t(p,x) & d_t(p,x) \\
    0 & 0 & 0 & 1\end{array}\right]$.
\end{center}
where $A_t(p, x), B_t(p, x), C_t(p, x)$ and $D_t(p, x)$ are $(n - 1) \times (n - 1)$ matrices, and the remaining maps are vectors of appropriate dimension. Consider the matrix
\begin{center}
 $\tilde{f}_t(p,x) = \left[\begin{array}{ccc} A_t(p, x) & B_t(p, x) \\  C_t(p, x) & D_t(p, x)\end{array}\right]   \in \mathrm{Sp}_{2n-2}(\mathbb{C})$
\end{center}
At $t=0$ we have from the continuous solution $\Psi_{L+2}(h_0(p, x)) = f(p, x)$, which implies $\tilde{f}_0 = I_{2n-2}$ and thus the holomorphic map
\begin{center}
 $\tilde{f}_1 : P \times X \to \mathrm{Sp}_{2n-2}(\mathbb{C})$  
\end{center}
is null-homotopic. Let $i$ be the canonical inclusion of $ \mathrm{Sp}_{2n-2}(\mathbb{C})$ into $ \mathrm{Sp}_{2n}(\mathbb{C})$ given by
\begin{center}
    $i(\tilde{f}_1(p,x)) = i \left( \left[\begin{array}{ccc} A_t(p, x) & B_t(p, x) \\  C_t(p, x) & D_t(p, x)\end{array}\right] \right) =\left[\begin{array}{cccc} A_{t}(p, x)  & 0 & B_{t}(p, x) & 0 \\  0 & 1 & 0 & 0 \\
    C_t(p,x) & 0 & D_t(p,x) & 0 \\
    0 & 0 & 0 & 1\end{array}\right] $
\end{center}
By the induction hypothesis, we have
\begin{center}
    $i(\tilde{f}_1(p,x)^{-1})=\left[\begin{array}{cccc} \tilde{A}(p, x)  & 0 & \tilde{B}(p, x) & 0 \\  0 & 1 & 0 & 0 \\
    \tilde{C}(p,x) & 0 & \tilde{D}(p,x) & 0 \\
    0 & 0 & 0 & 1\end{array}\right] $
\end{center}
is a finite product of holomorphic elementary symplectic matrices. \\
\\
Then the matrix $\Psi_{L+2}(h_1(p, x))f(p, x)^{-1}i(\tilde{f}_1(p, x)^{-1})$ 
\begin{center}
   $  \left[\begin{array}{cccc} I_{n-1}  & 0 & 0 & b_{1,1}(p,x) \\  a_1(p,x)\tilde{A}(p,x)+b_{2,1}(p,x)\tilde{C}(p,x) & 1 & a_1(p,x)\tilde{B}(p,x)+b_{2,1}(p,x)\tilde{D}(p,x) & b_{3,1}(p,x) \\
   0 & 0 & I_{n-1} & d_1(p,x) \\
    0 & 0 & 0 & 1\end{array}\right]$
\end{center}
is an elementary symplectic matrix, where
\begin{center}
    $a_1(p,x)\tilde{A}(p,x)+b_{2,1}(p,x)\tilde{C}(p,x) = -d_1(p,x)^T$
\end{center}
and
\begin{center}
    $a_1(p,x)\tilde{B}(p,x)+b_{2,1}(p,x)\tilde{D}(p,x) = b_{1,1}(p,x)^T$.
\end{center}
This implies $f(p, x)$ has a finite holomorphic symplectic factorization as required. 
\end{proof}
Examining the proof we observe the following.
\begin{rema}
  If $K_{cont}(n,d)$ is such that any holomorphic 
  symplectic matrix of size $2n$ can be factorized by $K_{cont}(n,d)$ continuous elementary symplectic matrices, then we have
  \begin{align}
      K(n,d) \le K_{cont}(n,d) + K(n-1,d) + 3,
    \end{align}   
   Since by Tavgen reduction \cite{HKS}
   \begin{align}
   \quad K_{cont} (n,d) \le K_{cont}(2,d). 
\end{align}
  we deduce $K(n,d) \le (n-1) (K_{cont}(2,d) + 3)$.
\end{rema}

\section{Stratified ellipticity of $\Phi_K$}\label{sec-stra}

Recall that $S_K$ denotes the singularity set of the map $\Phi_K$ described in Proposition \ref{SingSet}. From \cite[Theorem 3.5]{Scho} we have the following. 
\begin{lemm}
    For $n \ge 2, K \ge 3$, the submersion $\Phi_K : (\mathbb{C}^{\frac{n(n-1)}{2}} \times \mathbb{C}^{\frac{n(n+1)}{2}})^K \setminus S_K \to \mathbb{C}^{2n}\setminus \{0\}$ is surjective. 
\end{lemm}

We start the proof of Proposition \ref{reg-locus}, namely the stratified ellipticity of the submersion $\Phi_K$ outside the singularity set $S_K$. The proof uses the method in \cite{KI}. Over each stratum of $\mathbb{C}^{2n}\setminus \{0\}$ in an appropriate stratification, we are able to reduce the $2n$ defining polynomial equations of the fibers of $\Phi_K$ to one single polynomial equation $p(x_1, \cdots , x_m) = a$. It turns out that the polynomial $p$ will be no more than linear in each of the variables $x_i, i = 1, \dots, m$. Therefore, the vector fields 
\begin{align*}
    V_{ij,p} = \frac{\partial p}{\partial x_i}\frac{\partial}{\partial x_j} - \frac{\partial p}{\partial x_j}\frac{\partial}{\partial x_i}, \quad  1 \leq i < j \leq m
\end{align*}
will be globally integrable \cite[Lemma 5.2]{KI} and span the tangent space of the fiber $\{ p = a \}$ at all smooth points \cite[Lemma 5.3]{KI}. The dominating fiber spray is then constructed as in Remark \ref{VF-spray}.

Recall that
\begin{center}
       For $k$ even, $M^+(A_k,Z_k) =\left[\begin{array}{ccc} I & Z_k \\ 0 & I\end{array}\right] \left[\begin{array}{ccc} A^{-1}_k & 0 \\ 0 & A^T_k\end{array}\right]  $, 
    \end{center}
    and
    \begin{center}
    for $k$ odd, $M^-(A_k,Z_k)=\left[\begin{array}{ccc} I & 0 \\ Z_k & I \end{array}\right] \left[\begin{array}{ccc} A_k^{-T} & 0 \\ 0 & A_k \end{array}\right] $.
\end{center}
 For $K$ even, we have:
 \begin{center}
     $\Phi_K((A_1,Z_1),\cdots (A_K,Z_K)) = P^K =  e^T_{2n} M^-(A_1,Z_1) M^+(A_2,Z_2) \cdots M^-(A_{K-1},Z_{K-1}) M^+(A_K,Z_K)$.
 \end{center}
 Hence we have a recursive formula:
 \begin{center}
     $P^K = P^{K-1} M^+(A_K,Z_K)$
 \end{center}
Denote $P^K_f=(P^K_1,\cdots,P^K_n)$, $P^K_s=(P^K_{n+1},\cdots,P^K_{2n})$. Then the above equation can be writen as
\begin{center}
    $P^K \left[\begin{array}{ccc} A_K & 0 \\ 0 & A^{-T}_K\end{array}\right] = P^{K-1} \left[\begin{array}{ccc} I_n & Z_K \\ 0 & I_n\end{array}\right]$
\end{center}
and this is equivalent to
\begin{equation}\label{equa-2}
    \begin{cases}
    P^K_f A_K = P^{K-1}_{f} \\
      P^K_s A^{-T}_K = P^{K-1}_s + P^{K-1}_f Z_K 
\end{cases}
\end{equation}
similarly we have
\begin{center}
     $P^{K-1} = P^{K-2} M^-(A_{K-1},Z_{K-1})$
\end{center}
which is equivalent to
\begin{center}
    $P^{K-1} \left[\begin{array}{ccc} A^T_{K-1} & 0 \\ 0 & A^{-1}_{K-1}\end{array}\right] = P^{K-2} \left[\begin{array}{ccc} I_n & 0 \\ Z_{K-1} & I_n\end{array}\right]$
\end{center}
and
\begin{equation}\label{equa-3}
    \begin{cases}
    P^{K-1}_f A^T_{K-1} = P^{K-2}_f+P^{K-2}_s Z_{K-1}  \\
      P^{K-1}_s A^{-1}_{K-1} = P^{K-2}_{s}
\end{cases}
\end{equation}

We define the following stratification of $\mathbb{C}^{2n} \setminus \{0 \}$
\begin{align*}
    G_i &=  \lbrace (a,b) \in \mathbb{C}^n \times \mathbb{C}^n : a_1 =  \cdots = a_{i-1} =0, a_i \neq 0 \rbrace , \quad i = 1, \dots, n \\
    N_i &= \lbrace (0,b) \in \mathbb{C}^n\times \mathbb{C}^n: b_n = \cdots = b_{n-i+1} =  0, b_{n-i} \neq 0  \rbrace, \quad i = 0, \dots, n-1
\end{align*}

\begin{lemm}\label{red-1}
Let $(a,b)\in G_i \subset \mathbb{C}^{2n}\setminus \{0\}, i = 1, \dots, n$. Then the fiber $(P^K)^{-1}(a,b)$ given by the set of $2n$ equations $P^K(A,Z) = (a, b)$ can be described by one equation, namely $P^K_i = a_i$.
\end{lemm}
\begin{proof}
    Step 1: Observe that for $j>i$
    \begin{align*}
        P^{K-1}_j = P^K_f A_K e_j = P^K_j+\sum_{k=1}^{j-1} a_{kj,K}P_k^K 
        = P^K_j+a_{ij,K}P_i^K+\sum_{k=i+1}^{j-1} a_{kj,K}P_k^K. 
    \end{align*}
Over $G_i$ we have $P^{K}_i = a_i \neq 0$, thus we can use these equations to express $a_{ij,K}$
\begin{align*}
    P^{K-1}_j &= P^K_j + a_{ij,K} a_i + \sum_{k=i+1}^{j-1} a_{kj,K}P_k^K, \\
    \implies a_{ij,K} &=(P^{K-1}_j-P^K_j-\sum_{k=i+1}^{j-1} a_{kj,K}P_k^K)/a_i \quad \text{for } j = i+1, \dots, n.
\end{align*}
Step 2: We now use the second set of equations in (\ref{equa-2}) to express $n$ variables in $Z_K$. Over $G_i$, we have $P^{K-1}_i = P^K_i = a_i \neq 0$. Let $\alpha = P^K_s A^{-T}_K-P^{K-1}_s$.
By (\ref{equa-2}) we have $\alpha = P^{K-1}_f Z_K$, that is
\begin{center}
    $\alpha_j = \alpha e_j = P^{K-1}_f Z_K e_j = \sum^n_{l=1} P^{K-1}_l z_{lj,K} $ for $j = 1, \cdots,n$.
\end{center}
When $j \neq i$
\begin{center}
    $z_{ij,K}= (\alpha_j-\sum^n_{l=1,l\neq i}z_{lj,K}P^{K-1}_l)/a_i $
\end{center}
and
\begin{center}
   $z_{ii,K} = (\alpha_i-\sum^n_{l=1,l\neq i}z_{il,K}P^{K-1}_l)/a_i$ (by symmetry of $Z_K$)
\end{center}
and then we have
\begin{center}
    $z_{ii,K}= \frac{1}{a_i}(\alpha_i-\sum^n_{l=1,l\neq i}P^{K-1}_l(\frac{1}{a_i}(\alpha_i-\sum^n_{k=1,k\neq i}z_{ki,K}P^{K-1}_k)))$.
\end{center}
It remains to consider the following system of equations
\begin{equation}
    \begin{cases}
    P^K_1 = P^{K-1}_1=0 \\
    \quad \quad \, \vdots \\
      P^K_{i-1}= P^{K-1}_{i-1} =  0\\
      P^K_i = P^{K-1}_{i} = a_i \neq 0
\end{cases}
\end{equation}
Step 3: Now we use the first set of equation in (\ref{equa-3}) to express $(i-1)$ variables of $A_{K-1}$ as follows:
\begin{center}
    $P^{K-1}_f A^T_{K-1} = P^{K-2}_f+P^{K-2}_s Z_{K-1} =:  \beta$
\end{center}
and
\begin{center}
    $\beta_j =  P^{K-1}_f A^T_{K-1} e_j = P^{K-1}_j + \sum^n_{k=j+1} P^{K-1}_k a_{jk,K-1}$.
\end{center}
For $j =  1,2,\cdots,i-1$
\begin{align*}
    \beta_j &= P^{K-1}_j + a_i a_{ji, K-1} + \sum^n_{k=j+1, k \neq i} P^{K-1}_k a_{jk,K-1} \\
    \implies a_{ji,K-1}  &= (\beta_j-P^{K-1}_j-\sum^n_{k=j+1,k \neq i}P^{K-1}_ka_{jk,K-1})/a_i.
\end{align*}
\end{proof}

\begin{lemm}\label{red-2}
     Let $(0,b)\in N_i \subset \mathbb{C}^{2n}\setminus \{0\}, i = 0, \dots, n-1$. Then the fiber $(\Phi_K|_{U_K})^{-1}(0,b)$ given by the set of $2n$ equations $P^K(A,Z)=(0,b)$ can be described by one equation, namely $P^K_{2n-i}=b_{n-i}$.
\end{lemm}
\begin{proof}
Recall the inductive formula
\begin{align*}
    P^K = P^{K-1} M^+(A_K,Z_K) = P^{K-1} \left[\begin{array}{ccc} I & Z_K \\ 0 & I\end{array}\right] \left[\begin{array}{ccc} A^{-1}_K & 0 \\ 0 & A^T_K\end{array}\right].
\end{align*}
We consider the fiber of $(0,b)$ in $N_i$. Then (\ref{equa-2}) becomes
\begin{align*}
    \begin{cases}
    0 = P^{K-1}_{f} \\
    b A^{-T}_K = P^{K-1}_s  
    \end{cases}
\end{align*}
which is equivalent to
\begin{align} \label{P_f=0}
    \begin{cases}
    0 = P^{K-1}_{f}  \\
    b  = P^{K-1}_s  A^{T}_K
    \end{cases}
\end{align}
Writing out the second equation, we have
\[
    ( b_1, \dots, b_{n-i}, 0, \dots, 0) = P^{K-1}_s 
    \begin{bmatrix}
     1 & 0 & \cdots &0 \\  a_{12,K} & 1 & \ddots & 0 \\
    \vdots & \ddots & \ddots & \vdots \\
    a_{1n,K} & \cdots & a_{n-1,n,K} & 1
    \end{bmatrix}
\]
The last $n+i+1$ components yield
\begin{align} \label{n+i+1}
    P^{K-1}_{2n-i}=b_{n-i}, \, P^{K-1}_{2n-i+1}=0, \, \dots, P^{K-1}_{2n-1} = 0, \, P^{K-1}_{2n} = 0.
\end{align}
Since $b_{n-i}$ is nonzero, we can use the first $n-i-1$ equations to express the variables $a_{k,n-i,K}, k = 1, \dots, n-i-1$ in $A_K$ as in Step 1 of the proof of Lemma \ref{red-1}.

We write $P^{K-1} = P^{K-2} M^-(A_{K-1}, Z_{K-1})$
\begin{center} \label{P^K-1}
    $P^{K-1} = P^{K-2} \left[\begin{array}{ccc} I_n & 0 \\ Z_{K-1} & I_n\end{array}\right] \left[\begin{array}{ccc} A^{-T}_{K-1} & 0 \\ 0 & A_{K-1}\end{array}\right] = P^{K-2} \left[\begin{array}{ccc} A^{-T}_{K-1} & 0 \\ 0 & A_{K-1}\end{array}\right] \left[\begin{array}{ccc} I_n & 0 \\ \tilde{Z}_{K-1} & I_n\end{array}\right]$,
\end{center}
where $\tilde{Z}_{K-1}:= A_{K-1}^{-1} Z_{K-1} A_{K-1}^{-T}$ is also symmetric.

\begin{center}
    $P^{K-1} \left[\begin{array}{ccc} I_n & 0 \\ -\tilde{Z}_{K-1} & I_n\end{array}\right] = P^{K-2} \left[\begin{array}{ccc} A^{-T}_{K-1} & 0 \\ 0 & A_{K-1}\end{array}\right]$
\end{center}
equivalently
\begin{equation}
\begin{cases}
    P^{K-1}_f - P^{K-1}_s \tilde{Z}_{K-1} &= P^{K-2}_f A^{-T}_{K-1} \\
    P^{K-1}_{s} A^{-1}_{K-1} &= P^{K-2}_{s} 
\end{cases}  
\end{equation}
Therefore, the first set of equations from \eqref{P_f=0} and the equations \eqref{n+i+1} are equivalent to 
\begin{align*}
    \begin{cases}
        - P^{K-1}_s \tilde{Z}_{K-1} &= P^{K-2}_f A^{-T}_{K-1} \\
        P^{K-1}_{2n-i} &= b_{n-i} \\
        \sum^{j-1}_{k=1}d_{kj}P^{K-1}_{n+k} & = P^{K-2}_{n+j}, \quad j = n-i+1, \cdots, n  
    \end{cases}
\end{align*}
where we choose 
\begin{center}
     $A^{-1}_{K-1}  = \left[\begin{array}{cccc} 1 & d_{12} & \cdots & d_{1n} \\  0 & 1 & \ddots & \vdots \\
    \vdots & \ddots & \ddots & d_{n-1,n} \\
    0 & \cdots & 0 & 1\end{array}\right]$.
\end{center}
A closer inspection on the last set of equations, together with \eqref{n+i+1}, gives
\begin{align}
    d_{n-i,j} b_{n-i} + \sum^{n-i-1}_{k=1}d_{kj}P^{K-1}_{n+k} = P^{K-2}_{n+j}  \quad \text{ for } j = n-i+1, \dots, n. 
\end{align}
Since $b_{n-i}$ is nonzero, we can use these equations to express $d_{n-i,j}$
\[
     d_{n-i,j} =  (P^{K-2}_{n+j} - \sum^{n-i-1}_{k=1}d_{kj}P^{K-1}_{n+k})/b_{n-i}  \quad  \text{ for } j = n-i+1, \dots, n. 
\]
Since $P^{K-1}_s \neq 0$, the first set of equations $- P^{K-1}_s \tilde{Z}_{K-1} = P^{K-2}_f A^{-T}_{K-1}$ can be used to express $n$ variables in $\tilde{Z}_{K-1}$ as in Step 2 of proof of Lemma \ref{red-1}. 

The introduction of $\tilde{Z}_{K-1}$ and $d_{ij}, i>j$ is another choice of coordinates on the set of unitriangular $\tilde{\Omega}$-symplectic matrices. It allows to reduce the set of $2n$ equations in a simple way to one single equation. 
\end{proof}

For $K$ odd the reduction to one equation is very similar. The stratification of $\mathbb{C}^{2n} \setminus \{0\}$ is as follows:
\begin{align*}
    &\lbrace (a,b) \in \mathbb{C}^n\times \mathbb{C}^n: b_1 = \cdots = b_{i-1} =  0, b_{i} \neq 0  \rbrace \quad \text{ for } i = 1, \dots, n, \\
    &\lbrace (a,0) \in \mathbb{C}^n\times \mathbb{C}^n: a_n = \cdots = a_{n-i+1} =  0, a_{n-i} \neq 0  \rbrace \quad \text{ for } i = 0, \dots, n-1.
\end{align*}

\section{Optimal Factorization of elementary matrices}
In this section we prove that the elementary matrices can be written as a product of unitriangular symplectic matrices with respect to the standard symplectic form. Most importantly, the number of factors is independent of the size of the matrices, i.e., independent of $n$. 

Let $T^+_n(\mathbb{C})$ ($T^-_n(\mathbb{C})$) denote the set of upper (respectively lower) triangular matrices in $\mathrm{GL}_n(\mathbb{C})$, and $UT^+_n(\mathbb{C})$ ($UT^-_n(\mathbb{C})$) the set of upper (resp.\ lower) unitriangular matrices in $\mathrm{GL}_n(\mathbb{C})$. Denote by $Sym_n(\mathbb{C})$ the space of symmetric square matrices of size $n$. 

\begin{lemm}\label{fact-ele-lemm}
Let $A: X \to T^\pm_n(\mathbb{C})$ be a holomorphic map such that the diagonal entries $\lambda_1, \dots, \lambda_n$ are pairwise distinct constants. Then $A$ is holomorphically diagonalizable by a unitriangular matrix, i.e., there exists a holomorphic map $K: X \to UT^\pm_n(\mathbb{C})$ such that $A = K\Lambda K^{-1}$, $\Lambda = \mathrm{diag}(\lambda_1, \dots, \lambda_n)$.   
\end{lemm}
\begin{proof}
    Let us consider the case $A = (a_{ij})$ is upper triangular. 
    Let $K$ be an upper unitriangular matrix and $\alpha_{ij}, 1 \le i < j \le n$ the entries of $K$ which are above the diagonal. The set of linear equations given by $A K = K \Lambda$ can be uniquely solved for $\alpha_{i, i+1}$ using the equation at position $(i, i+1)$
    \[
        a_{i, i+1} + \lambda_{i+1} \alpha_{i, i+1} = \lambda_i \alpha_{i, i+1}.
    \]
    After having solved the first superdiagonal, continue with the second superdiagonal, and so on. 
\end{proof}

\begin{theo}\label{fact-ele}
    Given holomorphic maps $A: X \to UT^+_n(\mathbb{C})$ and $Z: X \to Sym_n(\mathbb{C})$, there exist holomorphic maps $Z_i, \tilde{Z}_i : X \to Sym_n(\mathbb{C}), i = 1, \dots, 7$ such that
    \begin{align*}
        M^-(A, Z) &= 
        \begin{bmatrix} I & 0 \\ Z & I   \end{bmatrix}
        \begin{bmatrix} A^{-T} & 0 \\ 0 & A  \end{bmatrix}
        =
        \begin{bmatrix} I & 0 \\ Z_1 & I   \end{bmatrix}
        \begin{bmatrix} I & Z_2 \\ 0 & I   \end{bmatrix}
        \cdots 
        \begin{bmatrix} I & 0 \\ Z_7 & I   \end{bmatrix} \\
        &=
        M^-(Z_1) M^+(Z_2) \cdots M^-(Z_7)
        \quad  \text{ and } \\
        M^+(A, Z) &= 
        \begin{bmatrix} I & Z \\ 0 & I   \end{bmatrix}
        \begin{bmatrix} A^{-1} & 0 \\ 0 & A^{T}  \end{bmatrix}
        =
        \begin{bmatrix} I & \tilde{Z}_1 \\ 0 & I   \end{bmatrix}
        \begin{bmatrix} I & 0 \\ \tilde{Z}_2 & I   \end{bmatrix}
        \cdots 
        \begin{bmatrix} I & \tilde{Z}_7 \\ 0 & I   \end{bmatrix} \\
        &=
        M^+(\tilde{Z}_1) M^-(\tilde{Z}_2) \cdots M^+(\tilde{Z}_7).
    \end{align*}
\end{theo}
\begin{proof}
First we make the diagonal entries of $A$ pairwise distinct by multiplying it with a constant diagonal matrix $\Lambda$ with such property
\begin{center}
    $\left[\begin{array}{ccc} A & 0 \\ 0 & A^{-T}\end{array}\right] = \left[\begin{array}{ccc} \Lambda & 0 \\ 0 & \Lambda^{-1}\end{array}\right]
    \left[\begin{array}{ccc} \Lambda^{-1}A & 0 \\ 0 & \Lambda A^{-T}\end{array}\right] 
    =  
    \left[\begin{array}{ccc} \Lambda & 0 \\ 0 & \Lambda^{-1}\end{array}\right]
    \left[\begin{array}{ccc} \Lambda^{-1}A & 0 \\ 0 & (\Lambda^{-1} A)^{-T}\end{array}\right]$.
\end{center}
The constant-valued factor decomposes as
\[
    \begin{bmatrix} \Lambda & 0 \\ 0 & \Lambda^{-1}   \end{bmatrix}
    = 
    \begin{bmatrix} I & 0 \\ -\Lambda^{-1}  & I  \end{bmatrix}
    \begin{bmatrix} I & \Lambda - I \\ 0 & I  \end{bmatrix}
    \begin{bmatrix} I & 0 \\ I & I  \end{bmatrix}
    \begin{bmatrix} I &  \Lambda^{-1} - I \\ 0 & I  \end{bmatrix}.
\]
By Lemma \ref{fact-ele-lemm} there exists a holomorphic map $K : X \to UT^+_n(\mathbb{C})$ so that $\Lambda^{-1}A = K\Lambda^{-1} K^{-1}$. 
Write $P_1 = KK^T$, $P_2 =  K^{-T} \Lambda K^{-1}$, then $\Lambda^{-1}A = P_1 P_2$ and $(\Lambda^{-1} A)^{-T} = P_1^{-1} P_2^{-1}$. 
As in \cite[Lemma 1]{JLX2}, we can write the second factor, starting with an upper unitriangular matrix, as
\begin{center}
    $\left[\begin{array}{ccc} P_1 P_2 & 0 \\ 0 & P_1^{-1} P_2^{-1}\end{array}\right] 
    =  
    \left[\begin{array}{ccc} I & -P_1 \\ 0 & I\end{array}\right]
    \left[\begin{array}{ccc} I & 0 \\ P_1^{-1} - P_2 & I\end{array}\right]
    \left[\begin{array}{ccc} I & P_2^{-1} \\ 0 & I \end{array}\right]
    \left[\begin{array}{ccc} I & 0 \\ P_2 P_1 P_2 - P_2 & I\end{array}\right]$.
\end{center}
Combining these two factorizations, taking into account the lower triangular form of the first factor in $M^-$, gives a decomposition of $M^-$ into $7$ factors. For $M^+$ observe that
\[
    \begin{bmatrix} \Lambda & 0 \\ 0 & \Lambda^{-1}   \end{bmatrix}
    = 
    \begin{bmatrix} I & \Lambda - I \\ 0 & I  \end{bmatrix}
    \begin{bmatrix} I & 0 \\ I & I  \end{bmatrix}
    \begin{bmatrix} I &  \Lambda^{-1} - I \\ 0 & I  \end{bmatrix}
    \begin{bmatrix} I & 0 \\ -\Lambda  & I  \end{bmatrix} \quad \text{ and }
\]
\begin{center}
    $\left[\begin{array}{ccc} P_1 P_2 & 0 \\ 0 & P_1^{-1} P_2^{-1} \end{array}\right] 
    =  
    \left[\begin{array}{ccc} I & 0 \\ P_1^{-1}P_2^{-1}P_1^{-1}-P_1^{-1} & I\end{array}\right]\left[\begin{array}{ccc} I & P_1 \\ 0 & I\end{array}\right]\left[\begin{array}{ccc} I & 0 \\ P_2-P_1^{-1} & I\end{array}\right]\left[\begin{array}{ccc} I & -P_2^{-1} \\ 0 & I \end{array}\right]$. \qedhere
\end{center}
\end{proof}

\begin{rema}\label{fact-num-theo1}
    (a) When $n=2$ or $n=3$ in Theorem \ref{fact-ele}, direct computation shows that $4$ instead of $7$ factors suffice: $M^-(A, Z) = 
        M^-(Z_1) M^+(Z_2) M^-(Z_3) M^+(Z_4)$ and $M^+(A, Z) = 
        M^+(\tilde{Z}_1) M^-(\tilde{Z}_2) M^+(\tilde{Z}_3) M^-(\tilde{Z}_4)$. \\
    (b) For $n \ge 4$, direct computation shows that $5$ factors are not sufficient. Thus the number of factors for factoring $M^\pm(A, Z)$ as a product of elementary symplectic matrices with respect to the standard symplectic form is at least $6$. 
\end{rema}

\begin{proof}[Proof of Theorem \ref{main-theo2}]
By Theorem \ref{fact-ele}, $\mathrm{K}_{symp}(n,d) \leq 7 \tilde{\mathrm{K}}_{symp}(n,d)$. 
Then by Remark \ref{fact-num-theo1} (b), $6 \leq \mathrm{K}_{symp}(n,d)$.
In the case $n \leq 3$, by Remark \ref{fact-num-theo1} (a), $\mathrm{K}_{symp}(n,d) \leq 4\tilde{\mathrm{K}}_{symp}(n,d)$. 
\end{proof}

\begin{proof}[Proof of Proposition \ref{bounds}]
    The inequality \eqref{K(N,d)} follows from \cite[Theorem 2.4]{HKS} and Theorem \ref{main-theo2}.
    By \cite[Theorem 1]{HKS}, 
    \[
        \tilde{K}_{symp}(n,1) = 4 \text{ and }\tilde{K}_{symp}(n,2) \le 5 \text{ for all } n.
    \]
    Together with the inequalities \eqref{K-Ktilde}, this shows the upper bounds in \eqref{K(n,1)n=2,3}, \eqref{K(n,1)n>=4}, \eqref{K(n,2)n=2,3} and \eqref{K(n,2)n>=4}. 
     The lower bounds in \eqref{K(n,1)n=2,3} and \eqref{K(n,2)n=2,3} hold even for matrices with constant entries \cite{JTZ}, while the lower bounds in \eqref{K(n,1)n>=4} and \eqref{K(n,2)n>=4} follow from Theorem \ref{main-theo2}. 
\end{proof}

\section*{Funding}

All authors were partially supported by Schweizerischer Nationalfonds grant 200021-207335.



\begin{thebibliography}{9}
\bibitem[For10]{For1} F. Forstneri{\v c}, \textit{The Oka principle for sections of stratified fiber bundles}, Pure Appl.
Math. Q. \textbf{6} (2010), no. 3, 843–874. 


\bibitem[For17]{For2} F. Forstneri{\v c}, \textit{Stein manifolds and holomorphic mappings}, Second edition, Springer, Cham, 2017.


\bibitem[Gro89]{Gro} M. Gromov, \textit{Oka’s principle for holomorphic sections of elliptic bundles}, J. Amer. Math. Soc. \textbf{2} (1989), no. 4, 851–897.

\bibitem[HKS22]{HKS}
G. Huang, F. Kutzschebauch and J. Schott. \textit{Factorization of holomorphic matrices and Kazhdan's property (T)}, Bull. Sci. Math. \textbf{190} (2024), 103376.


\bibitem[IK12]{KI}
B. Ivarsson, F. Kutzschebauch. \textit{Holomorphic factorization of mappings into }$SL_n(\mathbb{C})$. Ann. Math. \textbf{175} (2012), no. 1., 45 - 69.

\bibitem[IKL22]{IKL}
B. Ivarsson, F. Kutzschebauch, E. Løw. \textit{Holomorphic factorization of mappings into }$Sp_4(\mathbb{C})$, Anal. PDE, \textbf{16} (2023), no. 1, 233–277.

\bibitem[JTZ20]{JTZ} 
    P. Jin, Y. Tang, A. Zhu, 
    \textit{Unit triangular factorization of the matrix symplectic group}. SIAM J. Matrix Anal. Appl. \textbf{41} (2020), no. 4, 1630–1650.
    
\bibitem[JLX22]{JLX2}
    P. Jin, Z. Lin, B. Xiao, 
    \textit{Optimal unit triangular factorization of symplectic matrices}, Linear Algebra Appl. \textbf{650} (2022), 236–247.
   

\bibitem[Scho24]{Scho}
    J. Schott,
    \textit{Holomorphic factorization of mappings into the symplectic group}. J. Eur. Math. Soc. (2025), published online first.
    
\bibitem[VSS11]{VSS} 
    N. A. Vavilov, A. V. Smolenskiĭ, B. Sury, 
    \textit{Unitriangular factorizations of Chevalley groups}. 
    Zap. Nauchn. Sem. S.-Peterburg. Otdel. Mat. Inst. Steklov. (POMI) 
    \textbf{388} (2011), 17–47, 309–310; translation in
    J. Math. Sci. (N.Y.) \textbf{183} (2012), no. 5, 584–599.

\end{thebibliography}
\end{document}